\documentclass[10pt,twocolumn,twoside]{IEEEtran}
\title{\bf \huge A Market Mechanism for Virtual Inertia}
\author{Bala Kameshwar Poolla$^1$, Saverio Bolognani$^1$, Li Na$^2$,  Florian D\"orfler$^1$%
\thanks{This material is supported by ETH start-up funds, the SNF Assistant Professor Energy Grant \#160573, NSF CAREER 1553407, NSF 1608509, and ARPA-E NODES.}%
\thanks{$^1$ B.K. Poolla, S. Bolognani, and F. D\"orfler are with the Automatic Control Laboratory at the Swiss Federal Institute of Technology (ETH) Z\"urich, Switzerland. Emails: {\tt \{bpoolla,bsaverio,dorfler\}@ethz.ch} }%
\thanks{$^2$ Li Na is with the School of Engineering and Applied Sciences, Harvard University.
Email: {\tt lina@seas.harvard.edu}.} }

\usepackage[utf8]{inputenc}
\usepackage[T1]{fontenc}
\usepackage{enumitem}

\usepackage[final]{graphicx}
\usepackage{xcolor}
\usepackage[vlined,ruled]{algorithm2e}
\usepackage{subfigure}
\usepackage[font={small},labelfont=it]{caption}
\graphicspath{{./fig/}}	

\usepackage{cite}     
\usepackage{dsfont} 	

\usepackage[cmex10]{amsmath}
\usepackage{bm}
\usepackage{amsthm}
\usepackage{amsfonts}
\usepackage{amssymb}
\usepackage{mathrsfs}

\newtheoremstyle{bfnote}%
{}{}%
{\itshape}{}%
{\bfseries}{.}%
{ }%
{\thmname{#1}\thmnumber{ #2}\thmnote{ (#3)}}

\theoremstyle{bfnote}
\definecolor{m}{RGB}{0,56,167}

\newtheorem{theorem}{Theorem}
\newtheorem{lemma}{Lemma}

\newtheorem{remark}{Remark}

\DeclareSymbolFont{bbold}{U}{bbold}{m}{n}
\DeclareSymbolFontAlphabet{\mathbbold}{bbold}

\newcommand{\vectorones}[1][]{\mathds{1}_{#1}}
\newcommand{\vectorzeros}[1][]{\mathbbold{0}_{#1}}

\newcommand\oprocendsymbol{\hbox{$\square$}}
\newcommand\oprocend{\relax\ifmmode\else\unskip\hfill\fi\oprocendsymbol}

\usepackage{dblfloatfix} 
\usepackage{empheq}

\DeclareFontFamily{OT1}{pzc}{}
\DeclareFontShape{OT1}{pzc}{m}{it}{<-> s * [1.10] pzcmi7t}{}
\DeclareMathAlphabet{\mathpzc}{OT1}{pzc}{m}{it}

\usepackage{framed}
\newlength\height 
\newlength\fwidth

\usepackage{circuitikz}
\usepackage{pgfplots}
\usepackage{booktabs}
\usepackage[bbgreekl]{mathbbol}
\usepackage{colortbl}
\usepackage{ragged2e}
\usepackage{enumitem}
\usepackage{standalone}
\usepackage{tikz}
\usetikzlibrary{positioning, shapes}
\usetikzlibrary{decorations.pathreplacing}
\usetikzlibrary{arrows}

\usetikzlibrary{calc,trees,positioning,arrows,chains,shapes.geometric,%
    decorations.pathreplacing,decorations.pathmorphing,shapes,%
    matrix,shapes.symbols}

\tikzset{
>=stealth',
  punktchain/.style={rectangle, rounded corners, 
    draw=black, very thick,text width=10em, 
    minimum height=3em, text centered, on chain},
  line/.style={draw, thick, <-},
  element/.style={tape,top color=white,bottom color=blue!50!black!60!,
    minimum width=8em,draw=blue!40!black!90, very thick,
    text width=10em, minimum height=3.5em, text centered, on chain},
  every join/.style={->, thick,shorten >=1pt},
  decoration={brace},
  tuborg/.style={decorate},
  tubnode/.style={midway, right=2pt},
}

\begin{document}
\maketitle
\thispagestyle{empty}
\pagestyle{empty}

\begin{abstract}
One of the recognized principal issues brought along by the steadfast migration towards power electronic interfaced energy sources is the loss of rotational inertia. In conventional power systems, the inertia of the synchronous machines plays a crucial role in safeguarding against any drastic variations in frequency by acting as a buffer in the event of large and sudden power generation-demand imbalances. In future power electronic-based power systems, the same role can be played by strategically located virtual inertia devices. However, the question looms large as to how the system operators would procure and pay for these devices. In this article, we propose a market mechanism inspired by the ancillary service markets in power supply. We consider a linear network-reduced power system model along with a robust $\mathcal H_{2}$ performance metric penalizing the worst-case primary control effort. With a social welfare maximization problem for the system operator as a benchmark, we construct a market mechanism in which bids are invited from agents providing virtual inertia, who in turn are compensated via a Vickrey-Clarke-Groves payment rule. The resulting mechanism ensures truthful bidding to be the dominant bidding strategy and guarantees non-negative payoffs for the agents. A three-region case study is considered in simulations, and a comparison with a regulatory approach to the same problem is presented.
\end{abstract}


\section{Introduction}
The significance of the role that the inertia of large rotating synchronous machines played in conventional power grids has come to the fore once again. There has been a concerted effort to phase out fossil fuel-based generation in favour of power electronic interfaced solar and wind-based renewable generation. However, such a transition has been accompanied by a rise in the instances of reported grid frequency violations \cite{Fingrid:16}. The low system inertia is ill-equipped to arrest large frequency swings caused by power imbalances. Among the numerous attempts carried out to address this issue of loss of inertia, a considerable focus is being directed at virtual inertia emulation. The key feature of such a scheme is mimicking the effect that the rotational inertia in traditional generation had on the grid. This is achieved through advanced control strategies and state-of-the-art power electronic devices \cite{SN-SD-MCC:13,HB-TI-YM:14,SD-JA:13,BK-BJ-YZ-VG-PD-BMH-BH:17,TV-DVH:16}. Previously, studies have been conducted \cite{BKP-SB-FD:17,AU-TB-GA:14,MP-JWSP-BF:17,TSB-TL-DJH:15,TB-FD:17,DG-SB-BKP-FD:17} to ascertain the impact of spatial arrangement of these virtual inertia devices. A performance metric accounting for frequency violations due to possible disturbances is adopted for this purpose. The location of these devices across the grid significantly alters the system response to disturbances, and the optimal placement is found to heavily depend on the expected disturbance profile of the fluctuations that the grid may experience \cite{BKP-SB-FD:17}. However, the associated economics of such an integration has not received adequate attention. The pricing and payment of synchronous inertia find a mention in \cite{EE-GV-AT-BK-MM-MO:14}. The Australian Energy Market Commission security market report \cite{SH-AEMO:17} recommends virtual inertia provision by transmission operators and introducing market-sourcing mechanisms to this end.

The objective of this study is to identify potential conflicts arising due to the interaction between multiple non-cooperative stakeholders, e.g., end-consumers, transmission system operators, renewable energy source providers, etc., and propose a plausible solution which safeguards individual interests. To this end, the problems of provisioning of virtual inertia units and the affiliated payment architecture are considered within the framework of ancillary-service markets coupled with auction theory. Pricing as an instrument to assist frequency regulation is not new, having being discussed in \cite{ARW-FCS:89}. Numerous works \cite{JZ-KB:03,JAT-AN-DSC-KP:13,TT-AZC-CL:12} have weighed-in on markets for regulation in power systems. The authors in \cite{WT-RJ:16,WT-RJ:15} consider variants of the 

problem of mechanism design for electricity markets, while \cite{WL-EB:17} discusses a game-theoretic characterization of market power in energy markets. Among the several market mechanisms, energy markets based on the celebrated Vickrey-Clarke-Groves (VCG) mechanism \cite{PM:00,WV:61,TG:73} have been previously explored in \cite{PGS-NW-MK:17,YX-SL:17,WT-RJ:17}. Such a mechanism maximizes a pre-defined global cost (e.g., social welfare of all stakeholders), while allowing individual agents to operate as per their private interests. {However, to the best of our knowledge, there is a lack of market design for virtual inertia provision to ensure the system dynamical performance.}

{\it Contributions:} The contributions of this paper are as follows. {We construct and propose a framework wherein the procurement and pricing of virtual inertia across the power network, coupled with their effect on system performance is reformulated as an $\mathcal{H}_2$ system norm optimization problem.} To this end, we present two approaches: the centralized problem formulation-- which is considered as an efficiency benchmark and a market-based approach. Both these constructions cater to long-term planning as well as the day-ahead or the real-time (e.g., 5 minute) {procurement} scenarios. Though the centralized approach seems appealing due to its efficiency, it is quite idealistic, as all participating agents need to report their cost curves to the system operator. This exposes us to the concern of inflated reporting of cost curves, which often leads to unjustified profits for the agents. To overcome these shortcomings, we adopt the VCG framework from {mechanism design as the foundation of our proposed market mechanism}. This mechanism incentivizes the agents to participate (by guaranteeing non-negative payoffs) and to report their true cost curves--as the optimal bidding strategies. We also recover the efficiency of our benchmark solution, as the two procurement problems coincide, maximizing a social welfare objective. Though VCG has been well-studied, our mechanism is rendered novel due to the nature of the problem we address. {Our min-max problem setup with infinite-dimensional coupling constraint (an LTI system) is a non-trivial example for applying VCG.} The social welfare objective considered in this paper is a linear combination of costs accounting for the {worst-case} post-fault dynamic response of the power network and the costs of procuring virtual inertia. Moreover, because of the special structure of the problem, the VCG mechanism applied to our problem is proved to be individual rational i.e., every market participant is guaranteed to receive non-negative net revenue. Our analytic results rest on a deliberately stylized power system model and an $\mathcal H_{2}$ frequency performance criterion, though they can be extended (numerically) to more detailed models and other performance indices as long as the underlying optimization problems can be solved efficiently. {Finally, as the value of inertia and procurement thereof, is location-dependent, we show that this non-homogeneity prevents the existence of a global price for virtual inertia. Moreover, even with a large number of agents, it becomes critical to have them reveal the true cost and they do not possess enough market power unless they collude.}

The remainder of this section introduces some notation. Section \ref{Section: Problem Formulation} provides a background on the problem by revisiting some previously established results on $\mathcal{H}_2$ norm based coherency metrics. Section \ref{Section: Decentralised} introduces the different costs associated with virtual inertia and provides a benchmark procurement problem. Such a problem admits a solution which maximizes a measure of social welfare. Section \ref{Section: Inertia Proc} presents an efficient market-based setup for inertia procurement. Section \ref{Section: Case study} presents a case study and simulations on a three-region network, where we compare the market-based and classic regulatory approaches. Finally, Section \ref{Section: Conclusions} concludes the paper.
\paragraph*{Notation}
We denote the $n$-dimensional vectors of all ones and zeros by $\vectorones[n]$ and $\vectorzeros[n]$.
Given an index set $\mathcal{I}$ with cardinality $|\mathcal{I}|$ and a real-valued array $\{{x_1},\dots, x_{|\mathcal{I}|}\}$, we denote by $x \in \mathbb{R}^{|\mathcal{I}|}$ the vector obtained by stacking the scalars $x_i$ and by $\textup{diag}\{{x_i}\}$ the associated diagonal matrix.
The vector $\mathbbold e_{i}$ is the $i$-th vector of the canonical basis for $\mathbb{R}^n$.


\section{Problem Background}
\label{Section: Problem Formulation}

In this section, we present our power system model and integral-quadratic $\mathcal H_{2}$ performance criterion for frequency stability. We consider a power system model consisting of network-reduced swing equations accounting for the dynamics of synchronous machines and virtual inertia devices as in \cite{BKP-SB-FD:17,AU-TB-GA:14,MP-JWSP-BF:17,TSB-TL-DJH:15}. This coarse-grain model for the power system frequency dynamics together with the $\mathcal H_{2}$ metric allows us to derive our market mechanism in an analytic and insightful fashion. We remark that the results in this paper hold equally true when considering higher-order models of synchronous machines including governor dynamics, inertia-emulating power converters as in \cite{DG-SB-BKP-FD:17}, and also other performance metrics \cite{MP-JWSP-BF:17,TSB-TL-DJH:15,TB-FD:17} as long as the underlying optimization problems can be solved efficiently.%

\subsection{System Modeling}
\label{subsec:SysModel}
We consider a power network modeled by an underlying graph with nodes (buses) $\mathcal{V} =\{1, \dots, n\}$ and edges (transmission lines)  $\mathcal{E} \subseteq \mathcal{V} \times \mathcal{V}$. Next, a small-signal version of a network-reduced power system model \cite{PWS-MAP:98} is considered, where passive loads are modeled as constant impedances and eliminated via Kron reduction \cite{FD-FB:11d}, and the network is reduced to the sources $i \in \{1, $\dots$, n \}$ with linearized dynamics. The assumptions of identical unit voltage magnitudes, purely inductive lines, and a small signal approximation \cite{PWS-MAP:98} result in
\begin{equation}
\label{eq:basic}
m_{i} \ddot \theta_{i} + d_{i} \dot \theta_{i} = p_{\text{in}, i} - \sum_{\{i, j \}\in \mathcal{E}} b_{ij} (\theta_i -\theta_j),\quad \forall \,i,
\end{equation}
where ${p_{\text{in}, i}}$ refers to the net (i.e., electrical and mechanical) nodal power input, $b_{ij} \geq 0$ is the inverse of the reactance between nodes $\{i,j\} \in \mathcal E$. If bus $i$ only hosts a single synchronous machine, then \eqref{eq:basic} describes the electromechanical swing dynamics for the generator rotor angle $\theta_{i}$ \cite{PWS-MAP:98}, $m_{i}>0$ is the generator's rotational inertia, and  $d_{i}>0$ accounts for frequency damping or primary speed droop control (neglecting ramping limits). 

The renewable energy sources interfaced via power electronic inverters \cite{QCZ-TH:13} are {also} compatible with this setup. For such an interconnection, $\theta_{i}$ is the voltage phase angle, $d_{i}>0$ is the droop control coefficient, and $m_{i}>0$ either accounts for power measurement time constant \cite{JS-DG-JR-TS:13}, or arises from virtual inertia emulation through a dedicated controlled device~\cite{SN-SD-MCC:13,HB-TI-YM:14,SD-JA:13}, or is simply a control gain \cite{IAH-EMF:08}. Finally, the dynamics \eqref{eq:basic} may also arise from frequency-dependent or actively controlled frequency-responsive loads~\cite{PK:94}.
In general, each bus $i$ may host an ensemble of these devices, and the quantities $m_i$, $d_i$ reflect their aggregate behavior.

We wish to characterize the response of the interconnected system to disturbances in the nodal power injections, possibly representing faults, disconnection of loads or generators, intermittent demand, or fluctuating power generation from renewable sources as in \cite{BKP-SB-FD:17,DG-SB-BKP-FD:17,AU-TB-GA:14,MP-JWSP-BF:17,TSB-TL-DJH:15,TB-FD:17,UL-AM-MM-PW:17}. To do so, we consider the linear power system model \eqref{eq:basic} driven by the inputs ${p_{\text{in}, i}}$ parametrized as ${p_{\text{in}, i}}={{\pi}_i}^{\frac{1}{2}} \eta_i$.
In this parametrization, $\eta_i$ is a normalized disturbance, and the diagonal matrix $\Pi={\textup{diag}}\{{\pi}_{i}\}$ encodes the information about the strength of the disturbance at different buses, and therefore, describes also the location of the disturbance. We can obtain the state-space model for this linearized system as
\begin{equation*}
\begin{bmatrix}
\dot{\theta}\\
\dot{\omega}\\
\end{bmatrix}
=
\underbrace{
\begin{bmatrix}
{0} & I \\
-{M}^{-1}L & -{M}^{-1}{D}\\
\end{bmatrix}
}_{\text{\normalsize {\it=\,A}}}
\begin{bmatrix}
\theta\\
\omega\\
\end{bmatrix}
 +
\underbrace {
\begin{bmatrix}
0 \\
M^{-1} {\Pi}^{1/2}
\end{bmatrix}
}_{\text{\normalsize {\it=\,B}}}
\eta,
\label{eq: input}
\end{equation*}
where $M={\textup{diag}}\{m_{i}\}$ and $D ={\textup{diag}}\{d_{i}\}$ are the diagonal matrices of inertia and damping/droop coefficients, and $L \in \mathbb{R}^{n \times n}$ is the symmetric network susceptance matrix. The states $x=(\theta,\omega) \in \mathbb{R}^{2n}$ are the stacked vectors of angles and frequencies (deviations from the nominal values). Note that $A [\vectorones[n]^\top \; \vectorzeros[n]^\top]^\top = \vectorzeros[2n]$ and thus, $[\vectorones[n]^\top \; \vectorzeros[n]^\top]^\top$ is the right eigenvector corresponding to the eigenvalue zero.
To gauge the robustness of a power system to disturbances, we study generalized energy functions as metrics \cite{DG-SB-BKP-FD:17}, e.g., a quadratic function expressed as the time-integral 
\begin{equation}
	{\sum_i}\int_0^\infty y^i(t)^\top y^i(t) \,\text{d}t\, ={\sum_i}\int_0^\infty [{C \, x(t)}^i]^\top[{C \, x(t)}]^i\,\text{d}t\,,
	\label{eq:energy}
\end{equation}
where $y^i(t)= {C \, x(t)}^i$ is the corresponding output for an input {$u_i(t)$}. Note that by construction, the matrix $C$ is such that $C [\vectorones[n]^\top \; \vectorzeros[n]^\top]^\top = \vectorzeros[2n]$.{\footnote{As the system matrix is not Hurwitz, this choice ensures that the uncontrollable mode of $A$ is also unobservable from $C$.}}

The above performance metric can be interpreted as the energy amplification of the output $y$ (resp., its steady-state total variance) for impulsive disturbances $\eta$ (resp., unit variance white noise in a stochastic setting) and is commonly referred to as the squared $\mathcal{H}_2$ norm. The suitability of this metric to describe the stability and robustness of the system post fault is discussed in \cite{BKP-SB-FD:17,DG-SB-BKP-FD:17}.

\subsection{$\mathcal{H}_2$ norm calculation}
In this subsection, we recall a tractable approach for computing the energy metric presented in the previous subsection. 
\begin{lemma}{{\bf ($\mathcal{H}_2$ norm via observability Gramian)}}
\label{lemma: Observability Gramian}
Consider the state-space system $\mathrm{G}(A,B,C)$ defined above, with $A [\vectorones[n]^\top \; \vectorzeros[n]^\top]^\top = C [\vectorones[n]^\top \; \vectorzeros[n]^\top]^\top= \vectorzeros[2n]$. The $\mathcal{H}_2$ norm, ${\mathrm{Y}}(m,{\pi})$ is given by%
\begin{equation}
\label{eq:ObservGram}
{\|\mathrm{G}\|}_{2}^2={\mathrm{Y}}(m,{\pi})=\textup{Trace}(B^{\top}P B)\, ,
\end{equation}
where $Q={C}^{\top}{C}$, $P \in \mathbb{R}^{2n \times 2n}$ is the observability Gramian, uniquely defined by the following Lyapunov equation and an additional constraint:
\begin{align}
\label{eq: Lyap}
& {PA}+{A}^{\top}{P}+Q=0 \,,\\
\label{eq: Lyap constraint}
& P [\vectorones[n]^\top \; \vectorzeros[n]^\top]^\top = \vectorzeros[2n] \,.
\end{align}
\end{lemma}
It is known that the existence of the solution to \eqref{eq: Lyap} depends on the positive semi-definiteness of ${Q}$. Furthermore, from \cite{BKP-SB-FD:17}, we conclude that the solution $P$ is unique under the constraint \eqref{eq: Lyap constraint}. This result generalizes the uniqueness of $P$, as in \cite{BKP-SB-FD:17}, to any positive semidefinite matrix ${Q}$. Lemma \ref{lemma: Observability Gramian} allows us to analyze the generalized energy function \eqref{eq:energy} via an elegant $\mathcal{H}_2$ norm optimization problem. 

\subsection{Virtual inertia and post-fault behavior}
The minimization of ${\mathrm{Y}}(m,{\pi})$ through suitable choice of inertia coefficients $m_i$, attenuates the energy amplifications due to disturbances. {We stress here that as the $\mathcal{H}_2$ norm is a function of both the spatial distribution of the inertia $m$ in the grid and the location and strengths of the disturbances ${\pi}$  (through the input matrix $B$ and  through the cost function $C$), the effect of the location also impacts the market mechanism.} 

{\textbullet\ \bf Generalised energy function} -- The performance metric \eqref{eq:energy} computed in \eqref{eq:ObservGram}-\eqref{eq: Lyap constraint} as ${\mathrm{Y}}(m,{\pi})$ is generally non-convex in the inertia variable $m$ (cf. \cite{BKP-SB-FD:17}). {This is troublesome as it significantly complicates the analysis and study of the consequences of loss in synchronous inertia in large-scale power systems.}

A meaningful metric which singles out the effect of inertia on the post-fault frequency response is {desirable in order to quantify and mitigate any obstacles arising due to the loss in synchronous inertia}. To tackle these two concerns simultaneously, we consider the {output $y(t)=D^{1/2}\omega(t)$} as the cost in \eqref{eq:energy}. This yields
\begin{equation}
	{\mathrm{Y}_D} (m,{\pi})=\sum_i \mathlarger{\int_0^\infty}  [\omega(t)^\top D\, \omega(t)]^i\, {\text{d}t}\,,
	\label{eq:primaryeffort}
\end{equation}
which penalizes the frequency excursions at each node via the primary control effort $d \omega$ in \eqref{eq:basic} for disturbances at all inputs $u_i(t)$. This cost function is also justified from a system operator's viewpoint as significant resources are employed to contain frequency violations through droop control to restore frequency stability, which is effectively captured here. We note that a penalty on primary control effort is also (and especially) meaningful when considering more detailed higher-order power system models \cite{DG-SB-BKP-FD:17} where the optimal control strategy trades-off already existing droop control and additional costly inertia emulation. 

{\textbullet\ \bf Primary Control Effort} -- The performance metric $\mathrm{Y}_D (m,{\pi})$ in \eqref{eq:primaryeffort}, further, as a special case, admits a closed-form expression of the squared $\mathcal{H}_2$ norm via Lemma \ref{lemma: Observability Gramian} which is convex in the inertia variable $m$, i.e.,
\begin{equation}
\label{eq:primary effort trace}
{\|\mathrm{G}\|}_{2}^2={\mathrm{Y}_D}(m,{\pi})=\sum\nolimits_{i}\dfrac{{\pi}_i}{m_i},
\end{equation}
where ${\pi}_i$, $m_i$ are the disturbance strength, inertia at node $i$ respectively. {Note that the performance metric is linear in the disturbance strengths $\pi_i$.} We refer the reader to \cite{BKP-SB-FD:17} for a detailed proof.

As discussed in Section~\ref{subsec:SysModel}, ${\Pi}$ encodes the specific location of a disturbance besides the strength. However, as grid specifications necessitate performance guarantees against all possible contingencies, it is appropriate to consider a robust performance metric accounting for the {\it worst-case} disturbance. We can incorporate such {worst-case} requirements by exploiting the linearity of the performance metric \eqref{eq:primary effort trace} in ${\pi}$. We denote by ${\mathcal{P}}$, the set collecting all the possibly occurring disturbances (available from historical data, or forecasts). As a representative, we consider the following normalized set for the rest of the paper
\begin{equation}
{\mathcal{P}}:{\pi}\in \mathbb{R}_+^n,\, \vectorones[n]^{\sf T} {\pi} \leq {\pi}_\text{tot}.
\label{eq:distcons2}
\end{equation}
{This set $\mathcal{P}$ is a proxy for a set of bounded energy disturbances.} Let ${{\Gamma}}(m)$ be the {worst-case} performance, over the set of normalized disturbances in ${\mathcal{P}}$, i.e.,
\begin{equation}
{{\Gamma}}(m)=\underset{{\pi}\in {\mathcal{P}}} {\textup{max}} \quad{\mathrm{Y}_D}(m,{\pi}).
\label{eq:robust Q}
\end{equation}
As ${\mathrm{Y}_D}(m, {\pi})$ is linear in ${\pi}$, the maximization problem \eqref{eq:robust Q} can be reformulated as a tractable minimization problem, which for the primary control effort \eqref{eq:primary effort trace} is given by
\begin{subequations}%
\label{eq:robmin}%
\begin{empheq}[left = \text{${{\Gamma}}(m)=$}\empheqlbrace]{align}%
\underset{\rho\geq 0} {\textup{min}} & \quad {\pi}_\text{tot}\, {\rho} \\
\textup{subject to} 
& \quad m_i^{-1}-\rho \leq 0\,,\,\, \forall \,i\label{eq:robmin2}\,,
\end{empheq}%
\end{subequations}
where $\rho$ is the dual multiplier associated with the inequality budget constraint \eqref{eq:distcons2} on the disturbances.

{{\textbullet\ \bf Convex Upper-bounds for general cost functions} -- As we observed in \eqref{eq:primary effort trace}, the objective is convex in the inertia variable $m$, but this is not true for any general output matrix $C$. However, it is possible to obtain an upper-bound on the generalised energy function which is convex in $m$, as shown below in Remark~\ref{rem:ConvexBound}.

\begin{remark}[Convex upper-bound]\label{rem: convex bound}
\label{rem:ConvexBound}
\textup{
For generic performance metrics \eqref{eq:energy}, with an output matrix partitioned as $C=\textup{blkdiag}(C_1, \textup{diag}(C_2))$, a convex (in $m$) upper bound is \cite{BKP-SB-FD:17} 
\begin{equation*}
\label{eq:Bounds}
  {{\mathrm{Y}} (m, {\pi})}
\leq \overline{{\pi}}\cdot \underbrace{\frac{1}{2\underline{d}} \left\{\textup{Trace}({L}^\dagger\, Q_1) + \sum_{i=1}^n \frac{{(Q_2)}_{i}}{m_i}\right\}}_\text{\normalsize$=\mathcal{U}_b (m)$},
\end{equation*}
 where,
 $\overline{{\pi}} = {\textup{max}}_i{\{{\pi}_i}\}$, $\underline{d} = {\textup{min}}_i\{d_i\}$, $\,{L}^\dagger$ is the pseudo-inverse of the network Laplacian $L$, and the matrix $Q=C^\top C=\textup{blkdiag}(Q_1, Q_2)$. For the disturbance set $\mathcal{P}$ in \eqref{eq:distcons2}, the worst-case upper-bound $\mathcal{U}_b (m)$ on the performance $\mathrm{Y} (m,{\pi})$ in terms of the inertia variable $m$ is given by
\begin{equation*}
\mathcal{U}(m)=\underset{{\pi} \in {\mathcal{P}}} {\textup{max}} \quad \overline{{\pi}}\,\,\mathcal{U}_b(m)={\pi}_\textup{tot}\, \,\mathcal{U}_b (m).
\label{eq:rob upper}
\end{equation*}
Alternatively, this bound can be expressed as
\begin{empheq}[left = \text{${\mathcal{U}(m)}=$}\empheqlbrace]{equation*}%
\begin{split}
\underset{\rho\geq 0} {\textup{min}} & \quad {\pi}_\textup{tot}\, {\rho} \\
\textup{subject to} 
& \quad \mathcal{U}_b(m)-\rho\leq 0\,.
\end{split}
\end{empheq}%
}
\end{remark}

We note the striking similarity and parallels in the structures of $\mathcal{U}(m)$ and \eqref{eq:robmin}. In the subsequent sections we shall focus on the primary control effort for our analyses. However, we remark that this choice is not particularly restrictive, as investigating a convex upper bound (as in Remark \ref{rem: convex bound}) for generic cost functions \eqref{eq:energy} leads to a comparable analyses and analogous results-- due to convexity in the inertia variable $m$.}


\section{Centralized Planning Problem for Inertia}
\label{Section: Decentralised}
In this section, we motivate the inertia planning problem from the system operator's context. Consider a setting, wherein the system operator acquires virtual inertia units from multiple agents, with the provision of more than one agent per bus (or node). A possible regulatory solution involves sourcing units proportional to their capacity from each agent. This may, however, be uneconomical as virtual inertia units may be dissimilar based on the underlying technology and may not have identical costs. In order to make a more informed choice, it is, therefore, necessary to consider the costs entailed for virtual inertia provision by these units. 

The grid operator needs safeguarding against the {worst-case} limits \cite{ENTSOE2016} on $\|\dot\omega(t)\|_{\infty}$. Thus, a significant virtual inertia support in terms of large power injections for short durations of time post-fault is necessary to offset large frequency violations and arrest the rate of change of frequency $\dot \omega$. We recall from \eqref{eq:basic} that the power injected by virtual inertia devices during the initial post-fault transient is given by $p_\text{inj}\propto \|m \,\dot\omega(t)\|_{\infty}$. As such a service of providing high peak-power is not usually provided by the existing converters, it requires a dedicated oversized power converter interface. 

A previous study \cite{HT-CJ-AG:16} considered the problem of economic evaluation along with the design aspects of virtual inertia providing devices. The cost for provision of peak power was revealed as the overriding factor in determining the overall cost for such devices. Furthermore, it was proposed the normalised cost of physical inertia as a trading unit, rendering all the monetary costs proportional to $m$. In addition, agents would also incur some costs towards the maintenance of these resources over their lifetime. Together, all these contribute to the net cost of providing virtual inertia as a service. We consider $c(m)$, a convex non-decreasing function to account for the lumped cost of $m$.

The system operator, now equipped with the true cost information, determines the virtual inertia allocation by optimizing system performance in an economical manner. Mathematically, this translates to solving the following optimization problem
\begin{subequations}
\label{eq:Cent}
\begin{align}
\underset{\mu} {\textup{min}} \quad & \gamma\,{\Gamma}(m(\mu)) +\sum_{k\in \mathcal{A}}\,c_k(\mu_k) \label{eq:Cent_a} \\
\textup{subject to} \quad & 0\leq \mu_k \leq{\overline{\mu_k}}  \label{eq:Cent_c}, \quad \forall\,k,
\end{align}
\end{subequations}
where, ${\Gamma}(m)$ is the {worst-case} performance metric as in \eqref{eq:robmin}, $\mathcal{A}=\mathcal{A}_1\cup \ldots \cup \mathcal{A}_n$, $\mathcal{A}_i$ is the total number of agents at node $i$, is the total number of agents, $\mu_k$ is the virtual inertia obtained from agent $k$, and $\gamma>0$ trades-off the coherency metric and the economic costs. {We further note that $m(\mu)$ can be broken down as 
\begin{equation}
m_i=m_i^0+\sum_{k\in\mathcal{A}_i} \mu_k \label{eq:Cent_d},\quad \forall \,i,
\end{equation}
where, $m^0$ is the stacked vector of existing/residual (after the synchronous machines are replaced) inertia at the nodes and $\mu_k$ is the procured virtual inertia from each agent $k$. We model the capacity constraints of each agent via \eqref{eq:Cent_c}, where $\overline{\mu_k}$ is the maximum inertia procurement capacity from agent $k$. This set up is schematically represented in Figure~\ref{fig:setup}\footnote{For the centralised problem which is considered a benchmark, the agents only provide the costs $c(\mu_k)$ and not the bids $b(\mu_k)$.}. In the following, we use $\mu\in \mathcal{M}$ as a shorthand for the constraints encoded in \eqref{eq:Cent_c}.
}

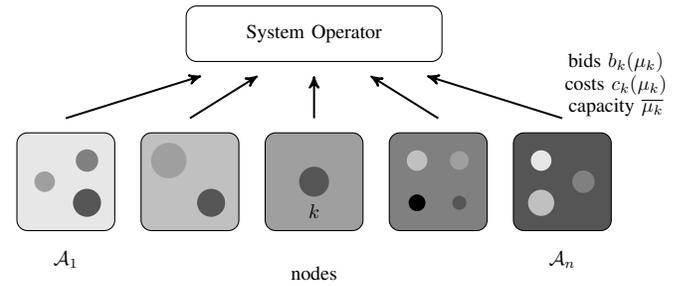
\begin{figure}[htbp]
\centering
\definecolor{lightgrayT}{RGB}{231,231,231}
\definecolor{silverT}{RGB}{192,192,192}
\definecolor{darkgrayT}{RGB}{159,159,159}
\definecolor{grayT}{RGB}{128,128,128}
\definecolor{dimgrayT}{RGB}{85,85,85}

\begin{tikzpicture}[scale=0.75, transform shape]
\tikzstyle{verylargebox} = [draw, rounded corners, text width=4.3cm, text height=0.75cm]

\tikzstyle{windbox} = [draw, rounded corners, text width=1.5cm, text height=1.5cm]
\tikzstyle{hydroboxl} = [draw, rounded corners, text width=1.5cm, text height=1.5cm]
\tikzstyle{gasbox} = [draw, rounded corners, text width=1.5cm, text height=1.5cm]
\tikzstyle{coalbox} = [draw, rounded corners, text width=1.5cm, text height=1.5cm]
\tikzstyle{oilbox} = [draw, rounded corners, text width=1.5cm, text height=1.5cm]

\coordinate (acord) at (0,0);
\coordinate[right=2.2cm of acord] (bcord);
\coordinate[right=2.2cm of bcord] (ccord);
\coordinate[right=2.2cm of ccord] (dcord);
\coordinate[right=2.2cm of dcord] (ecord);
\coordinate[above=3cm of acord] (fa);
\coordinate[right=4.4cm of fa] (fcord);

\node[windbox, fill=lightgrayT, anchor=south] (a) at (acord){};
\node[hydroboxl, fill=silverT, anchor=south] (b) at (bcord){};
\node[gasbox, fill=darkgrayT, anchor=south] (c) at (ccord){};
\node[coalbox, fill=grayT, anchor=south]  (d) at (dcord){};
\node[oilbox, fill=dimgrayT, anchor=south]  (e) at (ecord){};

\node[verylargebox, anchor=south] (f) at (fcord){};

\draw node[fill=dimgrayT, circle, scale=1.5] at (0.375, 0.5){};
\draw node[fill=grayT, circle, scale=1.2] at (0.375, 1.25){};
\draw node[fill=darkgrayT,circle, scale=1.1] at (-0.375, 0.875){};

\draw node[fill=dimgrayT, circle, scale=1.5] at (2.575,0.5){};
\draw node[fill=darkgrayT, circle, scale=1.9] at (1.825, 1.25){};

\draw node[fill=dimgrayT, circle, scale=1.6] at (4.4, 0.875){};

\draw node[fill=dimgrayT,circle, scale=0.75] at (6.975,0.5){};
\draw node[fill=silverT, circle, scale=1.1] at (6.225, 1.25){};
\draw node[fill=darkgrayT, circle, scale=1] at (6.975, 1.25){};
\draw node[fill,circle, scale=0.9] at (6.225, 0.5){};

\draw node[fill=grayT, circle, scale=1.2] at (9.175, 0.875){};
\draw node[fill=silverT, circle, scale=1.4] at (8.425, 0.5){};
\draw node[fill=lightgrayT, circle, scale=1.1] at (8.425, 1.25){};

\draw [line width=0.75pt, ->](0, 2)--(2.4, 2.75);
\draw [line width=0.75pt, ->](2.2, 2)--(3.4, 2.75);
\draw [line width=0.75pt, ->](4.4, 2)--(4.4, 2.75);
\draw [line width=0.75pt, ->](6.6, 2)--(5.4, 2.75);
\draw [line width=0.75pt, ->](8.8, 2)--(6.4, 2.75);

\node[] at (4.4,3.5) {{System Operator}};
\node[] at (4.4,-0.75) {{nodes}};
\node[] at (9.75,3.0) {{bids $b_k(\mu_k)$}};
\node[] at (9.75,2.6) {{costs $c_k(\mu_k)$}};
\node[] at (9.75,2.2) {{capacity $\overline{\mu_k}$}};
\node[] at (0,-0.5) {{$\mathcal{A}_1$}};
\node[] at (8.8,-0.5) {{$\mathcal{A}_n$}};
\node[] at(4.4, 0.35) {{$k$}};

\end{tikzpicture}
\caption{A schematic representing the operator-agent interaction for the centralised and market-based mechanism design and highlighting the richness of the problem under consideration. We depict different nodes $\mathcal{A}_i$ through rectangular boxes and their non-homogeneity in the value of inertia at these nodes through different colors. The number of circles in each rectangle indicate the number of agents, the color their dissimilarities and the radii of the circles, their inertia capacities.}
\label{fig:setup}
\end{figure}

The solution to \eqref{eq:Cent} is the allocation that maximizes the social welfare, by virtue of minimizing the {worst-case} coherency metric \eqref{eq:robmin} and the cost of virtual inertia. Such an allocation is inherently efficient and constitutes the benchmark solution to the virtual inertia procurement problem. The centralized problem \eqref{eq:Cent} and the variations thereof have been studied in \cite{BKP-SB-FD:17,DG-SB-BKP-FD:17,MP-JWSP-BF:17}. However, in liberalized energy markets, the cost functions $c_k(\mu_k)$ are private information of each agent and therefore the system operator can not directly solve \eqref{eq:Cent}. We consider such a framework in the next section. 

\begin{remark}[Planning and day-ahead scenarios]
\label{rem:planning}
\textup{
Both the centralized inertia procurement problem \eqref{eq:Cent} and the market-based approach (that we propose in Section \ref{Section: Inertia Proc}) can be considered on two operational time-scales. If analyzed in a long-term {planning framework}, virtual inertia is predominantly employed to counteract the ill-effects of integrating renewables, coupled with a planned phasing-out of synchronous machines. As the costs are biased towards high peak-power devices, a market is needed to  incentivise investment and installation of such devices. On the other hand, in a day-ahead scenario, virtual inertia is deployed for instantaneous frequency support in case of time-varying inertia profiles \cite{AU-TB-GA:14}, incidental re-dispatches, and anticipated fluctuations. This market can be employed at different time-scales, including real-time (e.g. 5 mins).}
\end{remark}


\section{A Market Mechanism for Inertia}
\label{Section: Inertia Proc}
 
In this section, we consider a market-based procurement approach, inspired by ancillary service markets. In such a setting, the power system regulator devises a market mechanism-- where the system operator invites bids for virtual inertia in lieu of a fair compensation to the agents providing the service. The additional burden is eventually borne by the end consumers. 
 The objectives that govern the design of such a mechanism are: safeguard against the agents benefiting from reporting inflated bids; ensure that the resulting payments incentivize agents to participate in the auction and assure non-negative returns; {and} guarantee the efficiency of the benchmark centralized planning problem discussed in Section \ref{Section: Decentralised}.

\subsection{Overview and Preliminaries: Mechanism Design}
\label{Subsection: Prelim}
We consider a mechanism design approach to design an auction for the system operator to procure virtual inertia at the buses. {Specifically, each agent with its private cost function $c_k(\cdot)$ is one agent. The system operator\\
 {\it (1)} collects a bid $b_k(\cdot)$ from each agent $k$,\\ 
 {\it (2)} determines an allocation/procurement $\mu_k(b)$, and \\
 {\it (3)} sets the payment $p_k(b)$ for each agent $k$. \\
 Here, $b=(b_1,\ldots, b_n)$ are the collected bids, and we alternatively use the notation $(b_k, b_{-k})$ to denote $b$. Each agent $k$ aims to choose its bidding curve $b_k(\mu_k)$ strategically, in order to maximize individual utility, which is a measure of profit. The utility function $u_k$, is computed as the difference between the payment and the investment costs, i.e., the utility evaluated when $\mu_k$ inertia units are provided by agent $k$ at a bid $b_k$ is given by
\begin{equation}
\label{eq:Utility}
u_k(b_k, b_{-k}) = {p_k(b_k, b_{-k})} - c_k(\mu_k(b_k, b_{-k})),
\end{equation}
and depends also on all other agents' bids.
 
 Different mechanism design approaches differ on the types of bids, allocation rule, and payment rule. Due to the Revelation Principle, in this paper, we only consider direct mechanism design, meaning that the bids $b_k$ are the same type of cost functions $c_k$.} In the following, we formalize a few of these notions via game-theoretic definitions \cite{TB-GJO:98} before proceeding to the main results.

{\textbullet\ \it Nash equilibrium}: In a non-cooperative game theoretic setting, a Nash equilibrium is a set of strategies for the players, such that a unilateral deviation by any participant (strategies of other players unchanged) does not result in any gain in the payoff. For the game under consideration, this translates to a set of bids $(b^\star_1, \ldots, b^\star_n)$ being a Nash equilibrium, if for all agents $k$, the utility $u_k(b_k, b_{-k})$ satisfies
$$
u_k(b^\star_k, b^\star_{-k}) \geq u_k(b_k, b^\star_{-k}),\quad \forall \, b_k,
$$
where the bids $b_k$, $b_{-k}$ are the bids of agent $k$ and other agents excluding $k$, respectively. 

{\textbullet\ \it Dominant strategy}: A strategy for any player is said to be dominant, if irrespective of the strategies of other players, the payoff is larger than any other of its strategies. This is a stronger result than a Nash equilibrium strategy, as it is alien to the strategies of other players.\\
For the game considered here, this translates to a bid $b^\star_k$ being dominant for agent $k$, if the utility is maximized, i.e.,
$$u_k(b^\star_k, b_{-k})\geq u_k(b_k, b_{-k}), \quad \forall \,b_{-k}, b_k.$$

{\textbullet\ \it Incentive-compatibility}: A mechanism is said to be {incentive-compatible} if every player maximizes its payoff by bidding true costs $b_k=c_k$, i.e., 
$$u_k(c_k, b_{-k})\geq u_k(b_k, b_{-k}), \quad \forall \,b_{-k}, b_k.$$

Observe that an incentive-compatible mechanism results in a dominant strategy (thus, also a Nash equilibrium) with {\em truthful} bids $b^\star_{k}=c_{k}$.


{\textbullet\ {\it Mechanism Structure}: From a mechanism design viewpoint, the key characteristics of the proposed market mechanism are as follows. The ``type space" corresponds to the set of convex, non-decreasing cost curves $c_k(\mu_k)$ which is the private information of each agent $k$. The ``message space" is the set of non-decreasing, convex bidding curves $b_k(\mu_k)$ such that $b_k(0)=0$ for each agent $k$. The mechanism thus designed is a direct mechanism from the Revelation Principle.} 
\subsection{VCG Market Mechanism}
\label{Subsection: VCG}
The auction theory literature is rich with mechanisms proposed to incentivize the participation of agents while prohibiting exorbitant bidding. A particularly popular one, which respects both the requirements, is the Vickrey-Clarke-Groves (VCG) mechanism \cite{PM:00,WV:61,TG:73}. In this paper, we assume that the regulator adopts this mechanism as a possible approach.

The system operator {sets up} an auction to procure virtual inertia for buses $i \in \mathcal{V}$. We denote by $k\in \mathcal{A}$, the different agents who participate in the auction {(see Figure~\ref{fig:setup})}. Each agent simultaneously submits non-decreasing and convex bidding curves $b_k(\mu_k)$ together with the margins on the maximum $\overline{\mu_k}$ virtual inertia that can be provided.

The system operator instead of \eqref{eq:Cent}, determines the allocation vector $\mu^\text{VCG}$, for individual bids $b_k(\mu_k)$ for each agent $k$,
by solving the optimization problem 
\begin{equation}
\mu^\text{VCG}(b) := \textup{arg} \underset{\mu \in \mathcal{M}} {\textup{min}} \quad\underbrace{\gamma\,{\Gamma}(m(\mu)) + \sum_{k\in \mathcal{A}}\,b_k(\mu_k)}_{\text{\normalsize $:=\mathcal{B}(\mu, b)$}}\,.
\label{eq:so_opt1}
\end{equation}

{We recall from {\eqref{eq:robmin}} that ${\Gamma}(m(\mu))$ is convex in $m$ (therefore, also convex in $\mu$). The objective $\mathcal{B}(\mu, b)$ being convex in the $\mu$'s, admits an optimal allocation which is unique and concurrently minimizes the {worst-case} coherency metric \eqref{eq:robmin}, while choosing the most economical bids.}

The agents $k \in \mathcal{A}$ receive a commensurate compensation for the virtual inertia $\mu_k$ they provide, according to the \emph{VCG payment rule} \eqref{eq:VCGpay}. The knowledge of the underlying mechanism is assumed to be known a priori to both the operator and agents. The payment made to agent $k$ translates to its {\it externality}, i.e., the difference of the system costs when agent $k$ is absent from the auction and the cost when agent $k$'s contribution is excluded.

Let $\mu^\text{VCG-$k$}$ be the vector of optimal allocations of agents, when agent $k$ abstains from the auction. This corresponds to the solution of the same optimization problem \eqref{eq:so_opt1}, with the same bids, however, with the constraint set $\mathcal{M}_{-k}:=\mathcal{M}\cap(\mu_k=0)$, accounting for the absence of $k$, i.e.,
\begin{equation}
\mu^\text{VCG-$k$}(b):= \textup{arg}\underset{\mu \in \mathcal{M}_{-k}} {\textup{min}}\quad \mathcal{B}(\mu, b).\label{eq:so_opt-i}
\end{equation}
The payment $p_k^\text{VCG}$ to agent $k$ is computed as
\begin{equation}
p_k^\text{VCG}=\Bigg\{\mathcal{B}(\mu^\text{VCG-$k$}, b)\Bigg\} -\Bigg\{\mathcal{B}(\mu^\text{VCG}, b) -b_k(\mu_k^\text{VCG})\Bigg\},
\label{eq:VCGpay}
\end{equation}
where we recall that $\mu^\text{VCG}$ and $\mu^\text{VCG-$k$}$ are defined in \eqref{eq:so_opt1} and \eqref{eq:so_opt-i} respectively.

The market mechanism is completely defined by the map that dictates the allocations $\mu_k$ and the payments $p_k$, given the bids $b_k$ collected from the agents. Figure~\ref{fig:mechanism} illustrates the operator-agent mechanism, where we make the allocations $\mu(b)$ and the payments $p(b)$ explicit functions of the bids.

\begin{figure}[h]
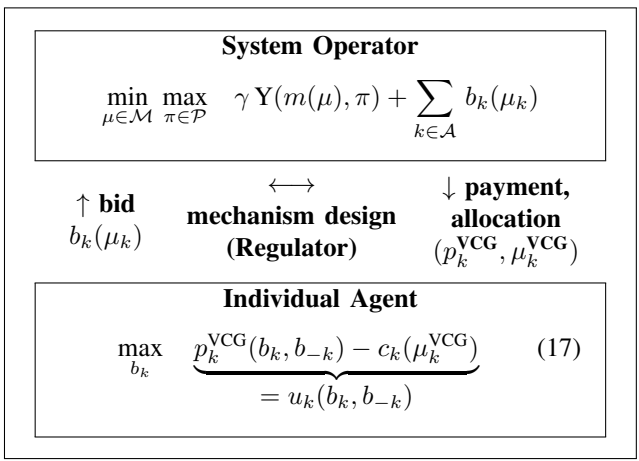

 \smallskip
\centering
\begin{minipage}{0.95\columnwidth}
\begin{framed}
\vspace{-.01in}
\begin{center}
\begin{minipage}{.98\columnwidth}
\begin{framed}
\begin{center}
\vspace{-.1in}
{\bf {System Operator}}
\begin{equation*}%
\label{eq:SysOp}
\underset{\mu \in \mathcal{M}}{\textup{min}}\,\,\underset{{\pi} \in {\mathcal{P}}}{\textup{max}}\quad \gamma\,{\text{Y}}(m(\mu), {\pi}) +\sum_{k\in \mathcal{A}}\,b_k(\mu_k)
\end{equation*}
\vspace{-.2in}
\end{center}
\end{framed}
\begin{minipage}{.25\textwidth}
\begin{center}
$\uparrow$ {\bf bid\\ $b_k(\mu_k)$}
\end{center}
\end{minipage}%
\begin{minipage}{.40\textwidth}
\begin{center}
$\longleftrightarrow$\\
{\bf mechanism design\\(Regulator)}
\end{center}
\end{minipage}%
\begin{minipage}{.35\textwidth}
\begin{center}
$\downarrow$ {\bf payment, allocation $(p_k^\text{VCG}, \mu_k^\text{VCG})$}
\end{center}
\end{minipage}

\begin{framed}
\begin{center}
\vspace{-.18in}
{\bf{Individual Agent}}
\begin{equation}
\label{eq:IndAg}
\underset{b_k}{\textup{max}}\quad \underbrace{{p_k^\text{VCG}(b_k, b_{-k})}-c_k(\mu_k^\text{VCG})}_{\text{\normalsize $=u_k(b_k, b_{-k})$}}
\end{equation}
\vspace{-.18in}
\end{center}
\end{framed}
\end{minipage}
\end{center}
\vspace{-.1in}
\end{framed}
\end{minipage}
\smallskip
\caption{Schematic representing the operator-agent mechanism through a regulator proposed mechanism design.}
\label{fig:mechanism}
\end{figure}

For the market setup in Figure~\ref{fig:mechanism}, we investigate suitable bidding strategies for the agents to maximize their utilities, in the theorem below. A mechanism additionally ensuring non-negative utilities is said to be {\it individually rational}. 
\begin{theorem}{{\bf(VCG-based inertia auction)}}
\label{theorem:dominant}
Consider the market setup described in Section~\ref{Subsection: VCG}, defined by the allocation function $\mu^\text{VCG}$ in terms of the bids \eqref{eq:so_opt1} and the payment function $p^\text{VCG}$ in \eqref{eq:VCGpay}. Then for every agent $k$,
\begin{enumerate}[label=(\roman*), leftmargin=0.05cm, itemindent=0.5cm]
\setlength{\itemsep}{2pt}
\item the VCG mechanism is incentive-compatible,
\item the utilities of the agents and payments made to the agents are non-negative,
\item and the allocation $\mu^\text{VCG}(c)$ is efficient, i.e., solves the social welfare maximization problem \eqref{eq:Cent}.
\end{enumerate}
\end{theorem}

\begin{proof}
Please refer to the Appendix.
\end{proof}
Theorem \ref{theorem:dominant} establishes that under the framework proposed in Figure \ref{fig:mechanism}, the agents maximize their respective utilities by bidding their true cost functions. Such a characteristic is desirable in order to prevent agents from accumulating profits by reporting inflated bids. We wish to highlight that this theorem adapts and extends the original VCG results to more complex {min-max} objectives with dynamic coupling (an $\mathcal{H}_2$ norm setting) to arrive at a payment scheme for virtual inertia.

{In addition, the proof also helps us to reflect on the assumptions made to arrive at the claimed result. We highlight that primarily, we relied on the convexity of the performance metric \eqref{eq:primary effort trace} (Section II. C) and the convexity of the bids (Section IV. A) in the inertia variable $m (\mu)$. On inspection, the proof relies on the ability to compute a global optimum of the combined cost function $\mathcal{B}(\mu, b)$. When considering large-scale non-linear power system models, computing the local minima for such a large dimensional problem is fairly easy, however it is very hard to compute the global minimum (even if it exists). As the VCG-based payments rely on recomputing the objective without the agents' contribution, it can be visualised that the numerically computed optimizer could converge to any of the multiple local minima (depending on the solver, initialisation, etc.) and therefore could have varied objective values at these minimizers. Under such scenarios, non-negativity of the utility function cannot be assumed or proved. Therefore, these assumptions can be relaxed if global minimum can be computed explicitly or numerically.}

\begin{remark}[Alternate problem formulation]
\label{rem:AlternateProb}
\textup{
The planning problem \eqref{eq:Cent} can alternatively be posed as-- minimizing costs subject to meeting certain performance guarantees. In such a framework, the optimization problem translates to
\begin{subequations}
\label{eq:Cent_alt}
\begin{align}
\underset{\mu \in \mathcal{M}} {\textup{min}} &\quad \sum_{k\in \mathcal{A}}\,c_k(\mu_k)\\
\textup{subject to} & \quad {\Gamma}(m(\mu)) \leq \overline{{\Gamma}}\label{eq:Cent_alt a}
\end{align}
\end{subequations}
where, $\overline{{\Gamma}}$ is a pre-specified performance guarantee which the system operator desires to meet. Observe that upon dualizing the constraint \eqref{eq:Cent_alt a} with a Lagrange multiplier $\gamma$, we obtain
\begin{equation}
\label{eq:Cent_alt2}
\underset{\gamma\geq0} {\textup{max}} \,\,\left\{  -\gamma\,\overline{{\Gamma}}+ \underset{\mu\in \mathcal{M}} {\textup{min}} \,\,\sum_{k\in \mathcal{A}}\,c_k(\mu_k)+\gamma\, {\Gamma}(m(\mu))\right\}.
\end{equation}
The inner minimization problem corresponds to the original problem formulation  \eqref{eq:Cent}. The outer maximization problem, on the other hand, is a tractable scalar optimization problem that can be solved via standard iterative procedures. As the optimal bidding strategy has been shown to be independent of $\gamma$ in Theorem~\ref{theorem:dominant}, the same bids can be used in all iterations of the outer maximization problem.}\end{remark}

\begin{remark}[Generality of the result]
\label{rem:Generality}
\textup{
Although we discuss a particular instance based on a coarse-grain swing equation model and an $\mathcal{H}_2$ performance metric here, Theorem~\ref{theorem:dominant} also holds true for more general system models such as the non-linear high-fidelity South East Australian system \cite{BKP-DG-FD:19}, as long as the $\textup{arg}\, \,\textup{min}$ of the expressions \eqref{eq:so_opt1}, \eqref{eq:so_opt-i} can be evaluated {efficiently}.}
\end{remark}


\section{Numerical Case Study}
\label{Section: Case study}
In this section, we present a few illustrations which support the preceding discussions on the centralized planning problem and the market-based auction mechanism. A 12-bus case study depicted in Figure~\ref{fig: sim} is considered. The system parameters are based on a modified two-region system from \cite[Example 12.6]{PK:94} with an additional third region, as introduced in \cite{TSB-TL-DJH:15}. We assume that each node is enabled to receive virtual inertia contributions from a number of dissimilar agents.
\subsection{Simulation setup}
We examine the alternate formulation presented in Remark~\ref{rem:AlternateProb} for the case study in Figure~\ref{fig: sim}. This setup enables us to impose guarantees on the {worst-case} performance via \eqref{eq:Cent_alt a}. Further, we note that the constraint \eqref{eq:Cent_alt a} requires solving a maximization problem over the disturbance set \eqref{eq:distcons2}. This can however be circumvented by rewriting the maximization inequality as a set of inequality constraints. We express \eqref{eq:Cent_alt a} from \eqref{eq:primary effort trace}, \eqref{eq:robust Q} as
\begin{subequations}
\label{eq:sim}
\begin{align}
\underset{{\pi}} {\textup{max}} &\quad \left\{ \sum_{i} \frac{{\pi}_i}{m_i} \right\}\leq \overline{{\Gamma}}\label{eq:sim a}\\
\textup{subject to} 
& \quad{{\pi}}_{i} \geq 0,\quad \forall\, i \label{eq:sim b}\\
& \quad \vectorones[n]^{\sf T} {\pi} \leq {\pi}_\text{tot} \label{eq:sim c}.
\end{align}
\end{subequations}
\begin{figure}[th]
\centering
\includegraphics[width=3.4in]{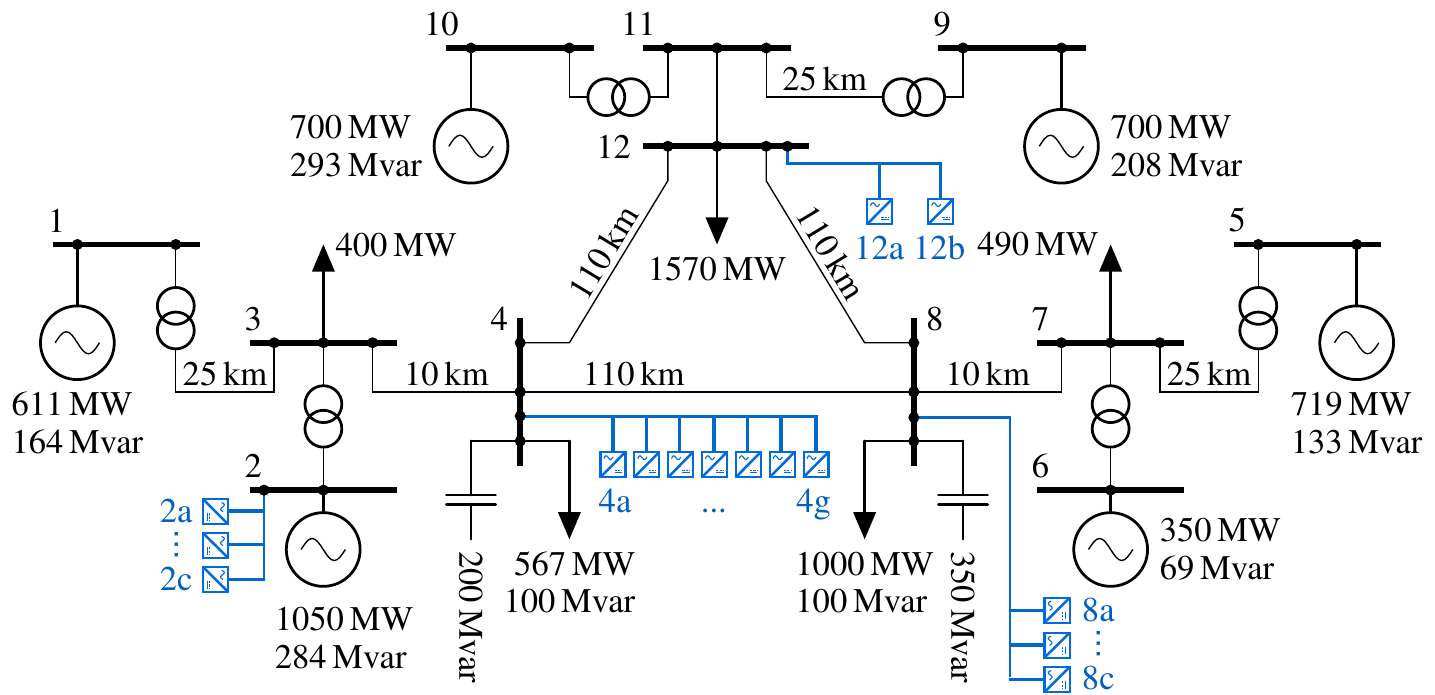}
\caption{A 12 bus three-region test case, grid parameters as in \cite{DG-SB-BKP-FD:17}.}
\label{fig: sim}
\end{figure}
The linearity of \eqref{eq:sim a} in ${\pi}$ dictates that the maximum can be obtained by evaluating the objective at the vertices of the polytope given by \eqref{eq:sim b}--\eqref{eq:sim c}, i.e., \eqref{eq:sim} is equivalent to the set of inequality constraints given by
\begin{equation}
\label{eq:simcons}
{{\pi}_\text{tot}} \leq \overline{{\Gamma}}\cdot {m_i}, \quad \forall \, i.
\end{equation}
The resulting set of constraints from \eqref{eq:simcons} is linear in the decision variables $\mu$ and can be directly incorporated in the optimization problem \eqref{eq:Cent_alt2}. For the purpose of simulations, we choose a non-trivial performance guarantee $\overline{{\Gamma}}=0.29$. This choice of $\overline{{\Gamma}}$ requires an additional virtual inertia allocation at the buses labeled $2, 4, 8, 12$ only, as the other buses by virtue of {$m^0$} already have sufficient inertia capability based on \eqref{eq:simcons} to meet this constraint specification.

\subsection{Simulation results}
The design of the market facilitates virtual inertia at each bus to be contributed by a number of agents. In our case study, we consider $3, 7, 3, 2$ agents connected to buses numbered $2, 4, 8, 12$ respectively, and label these as $(2\text{a},\ldots, 2\text{c})$, $(4\text{a}, \ldots, 4\text{g})$, $(8\text{a},\ldots, 8\text{c})$, $(12\text{a},\, 12\text{b})$. Further, we assume that the maximum inertia that each agent can contribute is given by the stacked vector $\overline{\mu}=[20,40,60,20,40,20,40,20,40,20,20,40,60,20,40]$. Note that the dissimilarity in the maximum inertia support can arise due to nature of the {underlying technology of the power electronic interface, the size of the energy device, maximum power rating} among others. We recall from Section~\ref{Subsection: VCG} that the bids and costs are expressed as non-decreasing convex curves. Here, we consider a special subset of linear functions, i.e., $c_k(\mu_k)=c_k\,\mu_k$ for the costs. As the costs incurred for provision of virtual inertia depend on the operating technology \cite{HT-CJ-AG:16}, we divide the costs into three monetary categories \footnote{The costs indicate the trend and do not necessarily reflect the exact cost of procurement.} for a qualitative analysis: {low}-- ($c=1$), {medium}-- ($c=5$), {high}-- ($c=10$). The true cost vector for the agents is chosen as $c^\text{sim}=[1,5,1,5,5,1,5,10,5,5,5,5,10,1,5]$, where the $c^\text{sim}_k$ element corresponds to the cost of procuring a unit inertia from agent $k$. We summarize our observations below.

\begin{figure}[ht]
\begin{subfigure}
\centering
%
%
\definecolor{mycolor1}{rgb}{0.49412,0.49412,0.49412}%
\definecolor{mycolor2}{HTML}{D59B2D}%
\definecolor{mycolor3}{HTML}{253F5B}%
\begin{tikzpicture}

\begin{axis}[%
width=2.2in,
height=1.4in,
at={(0.183in,1.945in)},
scale only axis,
bar shift auto,
xmin=0.511111111111111,
xmax=9.48888888888889,
xtick={1,2,3,4,5,6,7,8,9},
xticklabels={$1$, $2$, $4$, $5$, $6$, $8$, $9$, $10$, $12$},
xlabel style={font=\color{black}, yshift=1ex},
xlabel={\small{node}},
ymin=0,
ymax=50,
ytick={  0,  10,  20, 30, 40, 50},
ylabel style={font=\color{black},yshift=-3ex},
ylabel={\small{inertia}},
ymajorgrids,
axis background/.style={fill=white},
yticklabel style = {font=\footnotesize,xshift=0ex},
xticklabel style = {font=\footnotesize,yshift=0ex},
no markers,
every axis plot/.append style={ultra thin},
legend entries={$\underline{m}$, $m^\star$, $m_\text{uni}$},
legend style={legend cell align=left, align=left, draw=none, font=\small},
legend style={at={(0.45,0.97)}, anchor=north east}
]

\addplot[ybar, bar width=4, fill=mycolor1, draw=black, area legend] table[row sep=crcr] {%
1	37.24225668\\
2	12.41408556\\
3	7.219268219\\
4	35.38014385\\
5	35.38014385\\
6	12.73239545\\
7	37.24225668\\
8	37.24225668\\
9	19.98986085\\
};
\addlegendentry{$\, \underline{m}$}

\addplot[ybar, bar width=4, fill=mycolor2, draw=black, area legend] table[row sep=crcr] {%
1	37.242\\
2	34.483\\
3	34.483\\
4	35.38\\
5	35.38\\
6	34.483\\
7	37.242\\
8	37.242\\
9	34.483\\
};
\addlegendentry{$\, m^\star$}

\end{axis}

\begin{axis}[%
width=0.35in,
height=1.4in,
at={(2.883in,1.945in)},
scale only axis,
bar shift auto,
xmin=0.4,
xmax=1.8,
xtick={\empty},
ymin=0,
ymax=1.5,
ytick={0, 0.3,  0.6, 0.9, 1.2, 1.5},
yticklabels={$0$, $0.3$,  $0.6$, $0.9$, $1.2$, $1.5$},
ylabel style={font=\color{black},yshift=-3ex},
ylabel={\small{\text{worst-case} performance}},
axis background/.style={fill=white},
ymajorgrids,
yticklabel style = {font=\footnotesize,xshift=0ex},
xticklabel style = {font=\footnotesize,yshift=0ex},
no markers,
every axis plot/.append style={ultra thin},
legend style={legend cell align=left, align=left, draw=black}
]
\addplot[ybar, bar width=4, fill=mycolor1, draw=black, area legend] table[row sep=crcr] {%
1	1.2842\\
};

\addplot[ybar, bar width=4, fill=mycolor2, draw=black, area legend] table[row sep=crcr] {%
1.2	0.29\\
};

\end{axis}
\end{tikzpicture}%
\caption{Inertia profiles for the grid with a  primary control effort penalty and the associated {worst-case} performance.}
\label{fig: alloc} 
\end{subfigure}
\bigskip
\begin{subfigure}{}
%
%
\definecolor{mycolor1}{rgb}{0.49412,0.49412,0.49412}%
\definecolor{mycolor2}{HTML}{D59B2D}%
\definecolor{mycolor3}{HTML}{253F5B}%
\begin{tikzpicture}

\begin{axis}[%
width=3.0in,
height=1.4in,
at={(0.183in,4.95in)},
scale only axis,
bar shift auto,
xmin=0.5,
xmax=15.5,
xtick={1,2,3,4,5,6,7,8,9,10,11,12,13,14,15},
xticklabels={${2a}$,${2b}$,${2c}$,${4a}$,${4b}$,${4c}$,${4d}$,${4e}$,${4f}$,${4g}$,${8a}$,${8b}$,${8c}$,${12a}$,${12b}$},
xlabel style={font=\color{black}, yshift=1ex},
xlabel={\small{agents}},
ymin=0,
ymax=22,
ymajorgrids,
ylabel style={font=\color{black},yshift=-3ex},
ylabel={\small{inertia}},
axis background/.style={fill=white},
every axis plot/.append style={ultra thin},
yticklabel style = {font=\footnotesize,xshift=0ex},
xticklabel style = {font=\footnotesize,yshift=0ex},
legend style={legend cell align=left, align=left, draw=none, font=\small},
legend style={at={(0.8,0.97)},anchor=north east}
]
\addplot[ybar, bar width=3, fill=mycolor1, draw=black, area legend] table[row sep=crcr] {%
1	3.678\\
2	7.356\\
3	11.034\\
4	2.72635\\
5	5.4527\\
6	2.72635\\
7	5.4527\\
8	2.72635\\
9	5.4527\\
10	2.72635\\
11	3.625\\
12	7.25\\
13	10.875\\
14	4.782657\\
15	9.710243\\
};
\addplot[forget plot, color=white!15!black] table[row sep=crcr] {%
0.5	0\\
15.5	0\\
};
\addlegendentry{regulatory}

\addplot[ybar, bar width=3, fill=mycolor2, draw=black, area legend] table[row sep=crcr] {%
1	11.0343299243426\\
2	5.00009415801344e-09\\
3	11.0343287113236\\
4	1.45269215542908\\
5	1.45269154190225\\
6	19.9999999950002\\
7	1.4526934270501\\
8	3.9988693464494e-09\\
9	1.45268981506806\\
10	1.45269168624119\\
11	10.8751792487404\\
12	10.8751793719509\\
13	3.99888865669911e-09\\
14	14.4928586356894\\
15	4.99945381426038e-09\\
};
\addplot[forget plot, color=white!15!black] table[row sep=crcr] {%
0.5	0\\
15.5	0\\
};
\addlegendentry{market-based}

\addplot[ybar, bar width=3, fill=mycolor3, draw=black, area legend] table[row sep=crcr] {%
1	11.0343299243426\\
2	5.00009415801344e-09\\
3	11.0343287113236\\
4	1.45269215542908\\
5	1.45269154190225\\
6	19.9999999950002\\
7	1.4526934270501\\
8	3.9988693464494e-09\\
9	1.45268981506806\\
10	1.45269168624119\\
11	10.8751792487404\\
12	10.8751793719509\\
13	3.99888865669911e-09\\
14	14.4928586356894\\
15	4.99945381426038e-09\\
};
\addplot[forget plot, color=white!15!black] table[row sep=crcr] {%
0.5	0\\
15.5	0\\
};
\addlegendentry{centralized}

\end{axis}

\begin{axis}[%
width=3.0in,
height=0.3in,
at={(0.183in,4.2in)},
scale only axis,
bar shift auto,
xmin=0,
xmax=450,
xlabel style={font=\color{black}, yshift=1ex},
xlabel={\small{total cost (\$)}},
ymin=0.3,
ymax=2.3,
xmajorgrids,
ytick={\empty},
axis background/.style={fill=white},
yticklabel style = {font=\footnotesize,xshift=0ex},
every axis plot/.append style={ultra thin},
xticklabel style = {font=\footnotesize,yshift=0ex},
legend style={legend cell align=left, align=left, draw=black}
]
\addplot[xbar, bar width=3, fill=mycolor1, draw=black, area legend] table[row sep=crcr] {%
406.994722	2.4\\
};

\addplot[xbar, bar width=3, fill=mycolor2, draw=black, area legend] table[row sep=crcr] {%
201.630603628241	1.3\\
};

\addplot[xbar, bar width=3, fill=mycolor3, draw=black, area legend] table[row sep=crcr] {%
201.630603628241	0.2\\
};

\end{axis}
\end{tikzpicture}%
\caption{Virtual inertia contributions from agents and the total monetary cost incurred under regulatory, market-based setups.}
\label{fig: reg}
\end{subfigure}
\bigskip
\begin{subfigure}{}
%
%
\definecolor{mycolor1}{rgb}{0.49412,0.49412,0.49412}%
\definecolor{mycolor2}{HTML}{D59B2D}%
\definecolor{mycolor3}{HTML}{253F5B}%
\begin{tikzpicture}

\begin{axis}[%
width=3.0in,
height=1.4in,
at={(0.183in,0.45in)},
scale only axis,
bar shift auto,
xmin=0.514285714285714,
xmax=15.4857142857143,
xtick={1,2,3,4,5,6,7,8,9,10,11,12,13,14,15},
ytick={0, 1.5, 3, 4.5, 6, 7.5},
yticklabels={$0$, $1.5$, $3$, $4.5$, $6$, $7.5$},
xticklabels={${2a}$,${2b}$,${2c}$,${4a}$,${4b}$,${4c}$,${4d}$,${4e}$,${4f}$,${4g}$,${8a}$,${8b}$,${8c}$,${12a}$,${12b}$},
xlabel style={font=\color{white!15!black},yshift=1ex},
xlabel={\small{agents}},
ymin=0,
ymax=8,
ymajorgrids,
ylabel style={font=\color{black}, yshift=-3ex},
ylabel={\small{monetary units}},
axis background/.style={fill=white},
every axis plot/.append style={ultra thin},
yticklabel style = {font=\footnotesize,xshift=0ex},
xticklabel style = {font=\footnotesize,yshift=0ex},
legend style={legend cell align=left, align=left, draw=none, font=\small},
legend style={at={(0.55,0.97)},anchor=north east}
]
\addplot[ybar, bar width=4, fill=mycolor1, draw=black, area legend] table[row sep=crcr] {%
1	1.74989913655928\\
2	0\\
3	1.74989921899653\\
4	4.99999999999853\\
5	4.9999999999985\\
6	5.00000079200164\\
7	4.99999999999844\\
8	0\\
9	4.99999999999875\\
10	4.9999999999986\\
11	5.80475095785683\\
12	5.80475094875582\\
13	0\\
14	4.99999999448151\\
15	0\\
};
\addplot[forget plot, color=white!15!black] table[row sep=crcr] {%
0.514285714285714	0\\
15.4857142857143	0\\
};
\addlegendentry{Marginal payment}

\addplot[ybar, bar width=4, fill=mycolor2, draw=black, area legend] table[row sep=crcr] {%
1	1\\
2	0\\
3	1\\
4	5\\
5	5\\
6	1\\
7	5\\
8	0\\
9	5\\
10	5\\
11	5\\
12	5\\
13	0\\
14	1\\
15	0\\
};
\addplot[forget plot, color=white!15!black] table[row sep=crcr] {%
0.514285714285714	0\\
15.4857142857143	0\\
};
\addlegendentry{Marginal cost}

\end{axis}
\end{tikzpicture}%
\caption{Average payment and average cost profiles for agents participating in the market-based auction for procuring virtual inertia.}
\label{fig: pcu}
\end{subfigure}
\end{figure}

\begin{enumerate}[label=(\roman*), leftmargin=0.05cm, itemindent=0.5cm]
\setlength{\itemsep}{2pt}
\item In Figure~\ref{fig: alloc}, we plot the inertia profiles and the corresponding {worst-case} performance for the three-region case study. The robust optimal allocation problem yields a {\it valley-filling} profile as in \cite{BKP-SB-FD:17} which renders all buses identical with respect to the expected disturbance.
\item In Figure~\ref{fig: reg}, we compare the allocations and total procurement costs for three mechanisms: a possible {regulatory allocation} (where inertia is allocated proportionally to the capacity of the virtual inertia devices $\overline\mu$, in order to meet the specifications regardless of cost), market-based (Figure~\ref{fig:mechanism}), and the centralized \eqref{eq:Cent} mechanism. Each allocation is such that the performance requirements in \eqref{eq:sim} are met. The optimal allocations and the costs for the centralized and market-based mechanisms coincide. As the regulatory strategy does not factor for the cost-curves of the agents, it results in a higher overall cost.
\item Figure~\ref{fig: pcu} plots the average VCG-based payments (payment received per unit of virtual inertia) and average costs (cost per unit of virtual inertia) for each agent. Note that the agents $2$b, $4$e, $8$c, $12$b do not provide any virtual inertia due to higher per-unit cost of inertia. The average payment for each agent providing virtual inertia support and connected to the same node is identical. However, the cheapest agents are preferred-- this is reflected in payments larger than their costs (which is also their bid). Ultimately, the utility (the difference between payment and cost incurred) of each agent depends on the cost curves of all the agents that are co-located at the same bus.

\item {Another interesting observation is that the number of agents do not necessarily determine the extent of influence on pricing dynamics or market power, e.g., the agents $2$a, $2$c, $4$c, $12$a have the lowest cost of the service and yet they are not utilised to their capacity $\overline\mu$. This is due to the fact that system performance $\mathrm{Y}_D (m,{\pi})$ is both inertia and location dependent. This therefore also rules out the concept of a single-clearing price based on the cost and capacity.
}
 \end{enumerate}%

\section{Conclusions}
\label{Section: Conclusions}
We motivated the need for virtual inertia markets and considered the problem of economic, incentivized procurement of virtual inertia in low-inertia power grids. Two schemes for virtual inertia procurement were introduced and analyzed via an objective embedding a robust performance metric-- accounting for grid stability and the cost of virtual inertia procurement. {In contrast to a regulatory approach, the proposed market-based decentralized mechanism while respecting performance constraints, adequately compensated the agents for their contribution. Such a decentralized approach also resulted in a virtual inertia allocation which maximized the social welfare of the centralized planning problem.} This mechanism was partially inspired by the ancillary service markets and the payments to agents were based on the VCG rule. This enforced truthful bidding by the agents, achieved incentive compatibility, and non-negative utilities. We finally presented a few illustrations for a three-region case study underscoring the benefits of such a mechanism, while highlighting the peculiarities such as the absence of a single-clearing price, due to the heavy location-dependence of inertia as a service.

With more renewable integration in the power grids, virtual inertia is poised for a greater role, and the issues of economical procurement and markets for this service will gain further prominence. This paper is an attempt in this direction and we believe that such decentralized auction mechanisms for virtual inertia procurement will be integrated within the structure of existing power markets. In addition, we wish to point out that the decentralised mechanism presented here is model-agnostic. Though we illustrate our approach through a stylistic swing equation-based model, it can be extended to more complex system dynamics as long as convexity of the objective is established. Open problems not considered here pertain to robust auction mechanisms which counteract shill-bidding and collusion among various agents.
\section*{Acknowledgments}
The authors wish to thank Maryam Kamgarpour and Joe Warrington for their comments on the problem setup.

{\small
\renewcommand{\baselinestretch}{0.993}
\bibliographystyle{IEEEtran}
\bibliography{Pricing}

\begin{thebibliography}{10}
\providecommand{\url}[1]{#1}
\csname url@samestyle\endcsname
\providecommand{\newblock}{\relax}
\providecommand{\bibinfo}[2]{#2}
\providecommand{\BIBentrySTDinterwordspacing}{\spaceskip=0pt\relax}
\providecommand{\BIBentryALTinterwordstretchfactor}{4}
\providecommand{\BIBentryALTinterwordspacing}{\spaceskip=\fontdimen2\font plus
\BIBentryALTinterwordstretchfactor\fontdimen3\font minus
  \fontdimen4\font\relax}
\providecommand{\BIBforeignlanguage}[2]{{%
\expandafter\ifx\csname l@#1\endcsname\relax
\typeout{** WARNING: IEEEtran.bst: No hyphenation pattern has been}%
\typeout{** loaded for the language `#1'. Using the pattern for}%
\typeout{** the default language instead.}%
\else
\language=\csname l@#1\endcsname
\fi
#2}}
\providecommand{\BIBdecl}{\relax}
\BIBdecl

\bibitem{Fingrid:16}
J.~Jyrinsalo, ``Challenges and opportunities for the nordic power system,''
  Nordic Power System, Tech. Rep., August 2016.

\bibitem{SN-SD-MCC:13}
N.~Soni, S.~Doolla, and M.~C. Chandorkar, ``Improvement of transient response
  in microgrids using virtual inertia,'' \emph{IEEE Transactions on Power
  Delivery}, vol.~28, no.~3, 2013.

\bibitem{HB-TI-YM:14}
H.~Bevrani, T.~Ise, and Y.~Miura, ``Virtual synchronous generators: A survey
  and new perspectives,'' \emph{International Journal of Electrical Power \&
  Energy Systems}, vol.~54, 2014.

\bibitem{SD-JA:13}
S.~D'Arco and J.~Suul, ``Virtual synchronous machines -- classification of
  implementations and analysis of equivalence to droop controllers for
  microgrids,'' in \emph{IEEE POWERTECH}, 2013.

\bibitem{BK-BJ-YZ-VG-PD-BMH-BH:17}
B.~Kroposki, B.~Johnson, Y.~Zhang, V.~Gevorgian, P.~Denholm, B.-M. Hodge, and
  B.~Hannegan, ``Achieving a 100\% renewable grid: Operating electric power
  systems with extremely high levels of variable renewable energy,'' \emph{IEEE
  Power and Energy Magazine}, vol.~15, no.~2, 2017.

\bibitem{TV-DVH:16}
P.~Tielens and D.~V. Hertem, ``The relevance of inertia in power systems,''
  \emph{Renewable and Sustainable Energy Reviews}, vol.~55, pp. 999--1009,
  2016.

\bibitem{BKP-SB-FD:17}
B.~K. Poolla, S.~Bolognani, and F.~D{\"o}rfler, ``Optimal placement of virtual
  inertia in power grids,'' \emph{IEEE Transactions on Automatic Control},
  vol.~62, no.~12, pp. 6209--6220, 2017.

\bibitem{AU-TB-GA:14}
A.~Ulbig, T.~S. Borsche, and G.~Andersson, ``Impact of low rotational inertia
  on power system stability and operation,'' in \emph{IFAC World Congress},
  2014.

\bibitem{MP-JWSP-BF:17}
M.~Pirani, J.~W. Simpson-Porco, and B.~Fidan, ``System-theoretic performance
  metrics for low-inertia stability of power networks,'' \emph{arXiv preprint
  arXiv:1703.02646}, 2017.

\bibitem{TSB-TL-DJH:15}
T.~S. Borsche, T.~Liu, and D.~J. Hill, ``Effects of rotational inertia on power
  system damping and frequency transients,'' in \emph{54th IEEE Conference on
  Decision and Control}, 2015.

\bibitem{TB-FD:17}
T.~Borsche and F.~D{\"o}rfler, ``On placement of synthetic inertia with
  explicit time-domain constraints,'' \emph{IEEE Transactions on Power
  Systems}, 2017, {Submitted. Available at
  \url{https://arxiv.org/abs/1705.03244}}.

\bibitem{DG-SB-BKP-FD:17}
D.~Gro\ss, S.~Bolognani, B.~K. Poolla, and F.~D{\"o}rfler, ``Increasing the
  resilience of low-inertia power systems by virtual inertia and damping,'' in
  \emph{IREP Bulk Power System Dynamics \& Control Symposium}, 2017.

\bibitem{EE-GV-AT-BK-MM-MO:14}
E.~Ela, V.~Gevorgian, A.~Tuohy, B.~Kirby, M.~Milligan, and M.~O'Malley,
  ``Market designs for the primary frequency response ancillary service --
  {Part I}: Motivation and design,'' \emph{IEEE Transactions on Power Systems},
  vol.~29, no.~1, 2014.

\bibitem{SH-AEMO:17}
S.~Henry, ``System security market frameworks review, final report,'' AEMC,
  Tech. Rep., June 2017, Sydney.

\bibitem{ARW-FCS:89}
A.~W. Berger and F.~C. Schweppe, ``Real time pricing to assist in load
  frequency control,'' \emph{IEEE Transactions on Power Systems}, vol.~4,
  no.~3, pp. 920--926, 1989.

\bibitem{JZ-KB:03}
J.~Zhong and K.~Bhattacharya, ``Frequency linked pricing as an instrument for
  frequency regulation in deregulated electricity markets,'' in \emph{Power
  Engineering Society General Meeting, 2003, IEEE}, vol.~2.\hskip 1em plus
  0.5em minus 0.4em\relax IEEE, 2003, pp. 566--571.

\bibitem{JAT-AN-DSC-KP:13}
J.~A. Taylor, A.~Nayyar, D.~S. Callaway, and K.~Poolla, ``Consolidated dynamic
  pricing of power system regulation,'' \emph{IEEE Transactions on Power
  Systems}, vol.~28, no.~4, pp. 4692--4700, 2013.

\bibitem{TT-AZC-CL:12}
T.~Tanaka, A.~Z.~W. Cheng, and C.~Langbort, ``A dynamic pivot mechanism with
  application to real time pricing in power systems,'' in \emph{American
  Control Conference (ACC), 2012}.\hskip 1em plus 0.5em minus 0.4em\relax IEEE,
  2012, pp. 3705--3711.

\bibitem{WT-RJ:16}
W.~Tang and R.~Jain, ``Dynamic economic dispatch game: The value of storage,''
  \emph{IEEE Transactions on Smart Grid}, vol.~7, no.~5, pp. 2350--2358, 2016.

\bibitem{WT-RJ:15}
------, ``Market mechanisms for buying random wind,'' \emph{IEEE Transactions
  on Sustainable Energy}, vol.~6, no.~4, pp. 1615--1622, 2015.

\bibitem{WL-EB:17}
W.~Lin and E.~Bitar, ``A structural characterization of market power in power
  markets,'' \emph{arXiv preprint arXiv:1709.09302}, 2017.

\bibitem{PM:00}
P.~Milgrom, ``Putting auction theory to work: The simultaneous ascending
  auction,'' \emph{Journal of Political Economy}, vol. 108, no.~2, pp.
  245--272, 2000.

\bibitem{WV:61}
W.~Vickrey, ``Counterspeculation, auctions, and competitive sealed tenders,''
  \emph{Journal of Finance}, vol.~16, no.~1, pp. 8--37, 1961.

\bibitem{TG:73}
T.~Groves, ``Incentives in teams,'' \emph{Econometrica: Journal of the
  Econometric Society}, pp. 617--631, 1973.

\bibitem{PGS-NW-MK:17}
P.~G. Sessa, N.~Walton, and M.~Kamgarpour, ``Exploring the
  vickrey-clarke-groves mechanism for electricity markets,''
  \emph{IFAC-PapersOnLine}, vol.~50, no.~1, pp. 189--194, 2017.

\bibitem{YX-SL:17}
Y.~Xu and S.~H. Low, ``An efficient and incentive compatible mechanism for
  wholesale electricity markets,'' \emph{IEEE Transactions on Smart Grid},
  vol.~8, no.~1, pp. 128--138, 2017.

\bibitem{WT-RJ:17}
W.~Tang and R.~Jain, ``Aggregating correlated wind power with full surplus
  extraction,'' \emph{IEEE Transactions on Smart Grid}, 2017.

\bibitem{PWS-MAP:98}
P.~W. Sauer and M.~A. Pai, \emph{Power System Dynamics and Stability}.\hskip
  1em plus 0.5em minus 0.4em\relax Prentice Hall, 1998.

\bibitem{FD-FB:11d}
F.~D{\"o}rfler and F.~Bullo, ``{K}ron reduction of graphs with applications to
  electrical networks,'' \emph{IEEE Transactions on Circuits and Systems~I:
  Regular Papers}, vol.~60, no.~1, 2013.

\bibitem{QCZ-TH:13}
Q.-C. Zhong and T.~Hornik, \emph{Control of Power Inverters in Renewable Energy
  and Smart Grid Integration}.\hskip 1em plus 0.5em minus 0.4em\relax
  Wiley-IEEE Press, 2013.

\bibitem{JS-DG-JR-TS:13}
J.~Schiffer, D.~Goldin, J.~Raisch, and T.~Sezi, ``Synchronization of
  droop-controlled autonomous microgrids with distributed rotational and
  electronic generation,'' in \emph{{IEEE} Conf. on Decision and Control},
  2013.

\bibitem{IAH-EMF:08}
I.~A. Hiskens and E.~M. Fleming, ``Control of inverter-connected sources in
  autonomous microgrids,'' in \emph{American Control Conference}, 2008.

\bibitem{PK:94}
P.~Kundur, \emph{Power System Stability and Control}.\hskip 1em plus 0.5em
  minus 0.4em\relax McGraw-Hill, 1994.

\bibitem{UL-AM-MM-PW:17}
U.~M{\"u}nz, A.~Me{\v{s}}anovi{\'c}, M.~Metzger, and P.~Wolfrum, ``Robust
  optimal dispatch, secondary, and primary reserve allocation for power systems
  with uncertain load and generation,'' \emph{IEEE Transactions on Control
  Systems Technology}, 2017.

\bibitem{ENTSOE2016}
{RG-CE System Protection \& Dynamics Sub Group}, ``Frequency stability
  evaluation criteria for the synchronous zone of continental europe,''
  ENTSO-E, Tech. Rep., 2016.

\bibitem{HT-CJ-AG:16}
H.~Thiesen, C.~Jauch, and A.~Gloe, ``Design of a system substituting today's
  inherent inertia in the european continental synchronous area,''
  \emph{Energies}, vol.~9, no.~8, p. 582, 2016.

\bibitem{TB-GJO:98}
T.~Ba{\c{s}}ar and G.~J. Olsder, ``Dynamic noncooperative game theory,'' 1998.

\bibitem{BKP-DG-FD:19}
B.~K. Poolla, D.~Gro{\ss}, and F.~D{\"o}rfler, ``Placement and implementation
  of grid-forming and grid-following virtual inertia,'' \emph{IEEE Transactions
  on Power Systems}, 2019, to appear.

\end{thebibliography}
}

\section*{Appendix}
We present the proof sketch of Theorem 1 below.

\begin{proof}
In the following we use $\mu^\star$ as a shorthand for $\mu^\text{VCG}$, and $\hat{\mu}$ as a shorthand for $\mu^\text{VCG-k}$. We have from \eqref{eq:so_opt1}, 
\begin{equation}
\mu^\star(b)=\, \textup{arg}\,\, \underset{\mu \in \mathcal{M}}{\textup{min}}\quad \mathcal{B}(\mu, b), \label{eq:1}
\end{equation}
where, we recall that $\mu^\star(b)$=$\mu^\text{VCG}(b)$ is a function of the bids $b=(b_k, b_{-k})$. 

As before, let $\mu^\text{VCG-$k$}$ be the vector of optimal allocations of agents, when agent $k$ abstains from the auction. We have,
\begin{equation}
\hat{\mu}(b)=\, \textup{arg}\,\,\underset{\mu \in \mathcal{M}_{-k}} {\textup{min}} \quad\mathcal{B}(\mu,b),\label{eq:3}
\end{equation}
and $\hat{\mu}(b)$=$\mu^\text{VCG-$k$}(b)$ is a function of the bids $(b_{-k})$.

The individual payment to each agent $k$ as per the VCG payment rule \eqref{eq:VCGpay} is
\begin{equation*}
p^\text{VCG}_k=\Bigg\{\mathcal{B}(\hat{\mu}(b), b)\Bigg\} -\Bigg\{\mathcal{B}(\mu^\star(b), b) -b_k(\mu_k^\star(b))\Bigg\}.
\end{equation*}

Each agent maximizes individual utility with the local optimization problem for agent $k$ as in \eqref{eq:IndAg}. At the optimal allocation ($\mu^\text{VCG}, \, p^\text{VCG}$), with bids $(b_k,b_{-k})$, the utility of agent $k$ is
\begin{align}
&u_k(b_k, b_{-k})= {p^\text{VCG}_k(b_k, b_{-k})}-c_k(\mu^\text{VCG}_k(b_k, b_{-k}))\label{eq:UtilOpt1}\\
& =\mathcal{B}(\hat{\mu}(b),b)\underbrace{-\gamma\,{\Gamma}(m(\mu^\star(b)))-\sum_{j\neq k} b_j (\mu_j^\star(b))-c_k(\mu_k^\star(b))}_{\text{\normalsize{$=-\mathcal{B}(\mu^\star(b), (c_k, b_{-k}))$}}}\label{eq:UtilOpt2}.
\end{align}
We recall from \eqref{eq:Cent_d} that $m$ is a function of $\mu$. Hence, we can express $m(\mu^\star(b))$ as a function of $\mu^\star(b)$. Furthermore, we know from \eqref{eq:IndAg} that the agents desire to maximize individual utilities through an optimal bidding strategy $b_k$, preferably, independent of $b_{-k}$.

Note that the first term in \eqref{eq:UtilOpt2}, $\mathcal{B}(\hat{\mu}(b),b)$ is the objective evaluated at the optimizer obtained in \eqref{eq:3}. Via the constraint $\mu_k=0$ and bid characteristic $b_k(0)=0$, we conclude that the term $\mathcal{B}(\hat{\mu}(b),b)$ is independent of $(b_k,\mu_k)$. The utility $u_k$ in \eqref{eq:UtilOpt1} is therefore maximized when the second term $\mathcal{B}(\mu^\star(b), (c_k, b_{-k}))$ attains a minimum. Let $\mathcal{B}^\star(\mu^\star, (b_k^\dagger, b_{-k}))$ be the minimum of $\mathcal{B}(\mu^\star(b), (c_k, b_{-k}))$, i.e.,
\begin{equation*}
\label{eq:Opt_Agent_2}
\mathcal{B}^\star=\underset{b_k}{\textup{min}}\quad \gamma\,{\Gamma}(m(\mu^\star(b)))+\sum_{j\neq k} b_j (\mu_j^\star(b))+c_k(\mu_k^\star(b)).
\end{equation*}
From \eqref{eq:1}, note that
\begin{equation}
\label{eq:VCGp1}
\mathcal{B}(\mu^\star(b_k, b_{-k}), (b_k, b_{-k}))=\,\underset{\mu\in \mathcal{M}}{\textup{min}}\quad \mathcal{B}(\mu, (b_k, b_{-k})).
\end{equation}
Therefore, by changing arguments from $b_k$ to $c_k$ we get 
\begin{equation}
\label{eq:VCGp2}
\mu^\star(c_k, b_{-k})=\, \textup{arg}\,\underset{\mu\in \mathcal{M}}{\textup{min}}\quad \mathcal{B}(\mu, (c_k, b_{-k})),
\end{equation}
\begin{equation}
\label{eq:VCGp3}
\mathcal{B}(\mu^\star(c_k, b_{-k}), (c_k, b_{-k}))=\, \underset{\mu\in \mathcal{M}}{\textup{min}} \quad\mathcal{B}(\mu, (c_k, b_{-k})).
\end{equation}
From equations \eqref{eq:VCGp1}, \eqref{eq:VCGp2}, \eqref{eq:VCGp3}, and for the truthful bid $b_k(\mu_k)=c_k(\mu_k)$, the set $\mu^\star(b)=\mu^\star(b_k, b_{-k})$ is the minimizer of $\mathcal{B}(\mu^\star(b), (c_k, b_{-k}))$, with the minimum value of $\mathcal{B}^\star(\mu^\star, (b_k^\dagger, b_{-k}))=\mathcal{B}(\mu^\star(c_k, b_{-k}), (c_k, b_{-k}))$. 

The above argument applies to each agent $k$. Hence, bidding the true cost $b_k=c_k$ is a dominant strategy, and thus also a Nash equilibrium. Furthermore, as this bid $b_k=c_k$, is independent of the choice of the bids of other agents $b_{-k}$, such a mechanism is incentive-compatible.

The utility function of agent $k$ when all agents bid true costs, $u^\star_k(c_k, c_{-k})$, is
$$u^\star_k(c_k, c_{-k})=\mathcal{B}(\hat{\mu}(c),c)-\mathcal{B}({\mu^\star(c)}, c),$$ 
where $c=(c_k, c_{-k})$. As the first term $\mathcal{B}(\hat{\mu}(c),c)$ is evaluated over a smaller constraint set $\mathcal{M}_{-k}$, it is always larger than or equal to the second term $\mathcal{B}({\mu^\star(c)}, c)$. Consequently, it follows that bidding one's true cost also results in non-negative utilities $u_k^\star\geq0$ and thus also non-negative payments $p_k^\text{VCG}(c)\geq 0$ from \eqref{eq:IndAg}. As each agent bids the true cost, the problems \eqref{eq:so_opt1} and \eqref{eq:Cent} coincide. Hence, the optimal allocation $\mu^\text{VCG}(c)$ minimizes the social cost in the centralized problem \eqref{eq:Cent} and maximizes social welfare. This completes the proof.
\end{proof}

\end{document}